\documentclass[11pt]{article}
\usepackage{amsfonts,amscd,amssymb, amsmath}
\newtheorem{definition}{Definition}[section]

\newtheorem{proposition}[definition]{Proposition}

\newtheorem{remark}[definition]{Remark}

\newtheorem{theorem}[definition]{Theorem}

\def\rawo\lonra{\longrightarrow}

\def\ot{\otimes}

\newcommand{\selabel}[1]{\label{se:#1}}

\allowdisplaybreaks[4]

\newenvironment{proof}{{\it Proof.}}{\hfill $ \square $ \vskip 4mm}

\begin{document}
\title{Equivalent crossed products and cross product bialgebras
\thanks{This work was supported by a grant of the Romanian National 
Authority for Scientific Research, CNCS-UEFISCDI, 
project number PN-II-ID-PCE-2011-3-0635,  
contract nr. 253/5.10.2011.}}
\author {Florin Panaite\\
Institute of Mathematics of the 
Romanian Academy\\ 
PO-Box 1-764, RO-014700 Bucharest, Romania\\
e-mail: Florin.Panaite@imar.ro
}
\date{}
\maketitle
\begin{center}
{\em Dedicated to Fred Van Oystaeyen, on the occasion of his} 
$65^{\:th}$ {\em birthday}
\end{center}
\begin{abstract}
In a previous paper we proved a result of the type ''invariance under twisting'' for  
Brzezi\'{n}ski's crossed products. In this paper we prove a converse 
of this result, obtaining thus a characterization of what we call equivalent 
crossed products. As an application, we characterize cross product 
bialgebras (in the sense of Bespalov and Drabant) that are equivalent 
(in a certain sense) to a given cross product bialgebra in which one of the 
factors is a bialgebra and whose coalgebra structure is a tensor product 
coalgebra. 
\end{abstract}
\section*{Introduction}
${\;\;\;\;}$In \cite{brz}, Brzezi\'{n}ski introduced a very general construction, 
called {\em crossed product}, containing as particular cases several important 
constructions introduced before, such as twisted tensor products of algebras 
(as in \cite{Cap}, \cite{VanDaele}) and classical Hopf crossed products. 
Given an (associative unital) algebra $A$, a vector space $V$ endowed with a  
distinguished element $1_V$  and two linear maps $\sigma :V\otimes V
\rightarrow A\otimes V$ and $R:V\otimes A\rightarrow A\otimes V$ 
satisfying certain conditions, Brzezi\'{n}ski's crossed product is a certain 
(associative unital) algebra structure on $A\otimes V$, denoted in what follows by 
$A\otimes _{R, \sigma }V$. 

In \cite{panaite} we proved a result 
of the type invariance under 
twisting for  crossed products (containing as particular cases the 
invariance under twisting for twisted tensor products of algebras 
from \cite{jlpvo} and  
the invariance under twisting for quasi-Hopf smash products from 
\cite{bpvo}) as follows: if  $A\ot _{R, \sigma }V$ is a crossed product 
and $\theta , \gamma :V\rightarrow A\ot V$ are linear maps, 
we can define certain maps  $\sigma ':V\ot V\rightarrow A\ot V$ and 
$R':V\ot A\rightarrow A\ot V$ and if some conditions are 
satisfied then  $A\ot _{R', \sigma '}V$ is a crossed product, isomorphic 
to  $A\ot _{R, \sigma }V$.

Our first aim here is to prove a converse of this result.  First, we call two 
crossed products  $A\ot _{R, \sigma }V$ and $A\ot _{R', \sigma '}V$ 
{\em equivalent} if there exists a linear isomorphism 
$\varphi :A\ot _{R', \sigma '}V\simeq A\ot _{R, \sigma }V$ that is an 
algebra map and a morphism of left $A$-modules. With this terminology, our result says that 
two crossed products $A\ot _{R, \sigma }V$ and $A\ot _{R', \sigma '}V$ 
are equivalent if and only if there exist linear maps 
$\theta , \gamma :V\rightarrow A\otimes V$ satisfying a certain list of conditions. 

There exists a dual construction to Brzezi\'{n}ski's crossed product, called 
crossed coproduct. A {\em cross product bialgebra}, as defined by 
Bespalov and Drabant in \cite{BDII}, is a bialgebra whose algebra structure is a 
crossed product algebra and whose coalgebra structure is a crossed coproduct 
coalgebra. If  
$A _{\;W}^{\;\rho}\bowtie _{\;R}^{\;\sigma }C$ and 
$A _{\;W'}^{\;\rho '}\bowtie _{\;R'}^{\;\sigma '}C$ are cross product bialgebras, 
we call them {\em equivalent} if there exists a 
linear isomorphism 
$\varphi :A _{\;W'}^{\;\rho '}\bowtie _{\;R'}^{\;\sigma '}C
\simeq A _{\;W}^{\;\rho}\bowtie _{\;R}^{\;\sigma }C$ that is a 
morphism of bialgebras, of left $A$-modules and of right 
$C$-comodules. It is a natural problem to characterize all cross product 
bialgebras $A _{\;W'}^{\;\rho '}\bowtie _{\;R'}^{\;\sigma '}C$ that are 
equivalent to a given cross product bialgebra 
$A _{\;W}^{\;\rho}\bowtie _{\;R}^{\;\sigma }C$. We can solve here only a 
particular case of this problem, namely the case in which we assume that 
$A$ is a bialgebra and the coalgebra structure of the given cross product 
bialgebra is the tensor product coalgebra. 
\section{Preliminaries}\selabel{1}
${\;\;\;\;}$
We work over a commutative field $k$. All algebras, linear spaces
etc. will be over $k$; unadorned $\ot $ means $\ot_k$. 
By ''algebra'' (respectively ''coalgebra'') we 
always mean an associative unital algebra (respectively 
coassociative counital coalgebra). For the 
comultiplication of a coalgebra $C$, we use the version of 
Sweedler's sigma notation $\Delta (c)=c_1\ot c_2$ for 
all $c\in C$. 

We recall from \cite{Cap}, \cite{VanDaele} that, given two algebras $A$, $B$ 
and a $k$-linear map $R:B\ot A\rightarrow A\ot B$, with notation 
$R(b\ot a)=a_R\ot b_R$, for $a\in A$, $b\in B$, satisfying the conditions 
$a_R\otimes 1_R=a\otimes 1$, $1_R\otimes b_R=1\otimes b$, 
$(aa')_R\otimes b_R=a_Ra'_r\otimes b_{R_r}$, 
$a_R\otimes (bb')_R=a_{R_r}\otimes b_rb'_R$, 
for all $a, a'\in A$ and $b, b'\in B$ (where $r$ and $R$ are two different indices), 
if we define on $A\ot B$ a new multiplication, by 
$(a\ot b)(a'\ot b')=aa'_R\ot b_Rb'$, then this multiplication is associative 
with unit $1\ot 1$. In this case, the map $R$ is called 
a {\em twisting map} between $A$ and $B$ and the new algebra 
structure on $A\ot B$ is denoted by $A\ot _RB$ and called the 
{\em twisted tensor product} of $A$ and $B$ afforded by $R$.

We recall from \cite{brz} the construction of  
Brzezi\'{n}ski's crossed product:
\begin{proposition} (\cite{brz}) \label{defbrz}
Let $(A, \mu , 1_A)$ be an (associative unital) algebra and $V$ a 
vector space equipped with a distinguished element $1_V\in V$. Then 
the vector space $A\ot V$ is an associative algebra with unit $1_A\ot 1_V$ 
and whose multiplication has the property that $(a\ot 1_V)(b\ot v)=
ab\ot v$, for all $a, b\in A$ and $v\in V$, if and only if there exist 
linear maps $\sigma :V\ot V\rightarrow A\ot V$ and 
$R:V\ot A\rightarrow A\ot V$  satisfying the following conditions:
\begin{eqnarray}
&&R(1_V\ot a)=a\ot 1_V, \;\;\;R(v\ot 1_A)=1_A\ot v, \;\;\;\forall 
\;a\in A, \;v\in V, \label{brz1} \\
&&\sigma (1_V, v)=\sigma (v, 1_V)=1_A\ot v, \;\;\;\forall 
\;v\in V, \label{brz2} \\
&&R\circ (id_V\ot \mu )=(\mu \ot id_V)\circ (id_A\ot R)\circ (R\ot id_A), 
\label{brz3} \\
&&(\mu \ot id_V)\circ (id_A\ot \sigma )\circ (R\ot id_V)\circ 
(id_V\ot \sigma ) \nonumber \\
&&\;\;\;\;\;\;\;\;\;\;
=(\mu \ot id_V)\circ (id_A\ot \sigma )\circ (\sigma \ot id_V), \label{brz4} \\
&&(\mu \ot id_V)\circ (id_A\ot \sigma )\circ (R\ot id_V)\circ 
(id_V\ot R ) \nonumber \\
&&\;\;\;\;\;\;\;\;\;\;
=(\mu \ot id_V)\circ (id_A\ot R )\circ (\sigma \ot id_A). \label{brz5} 
\end{eqnarray}
If this is the case, the multiplication of $A\ot V$ is given explicitly by
\begin{eqnarray*} 
&&\mu _{A\ot V}=(\mu _2\ot id_V)\circ (id_A\ot id_A\ot \sigma )\circ 
(id_A\ot R\ot id_V),
\end{eqnarray*}
where $\mu _2=\mu \circ (id_A\ot \mu )=\mu \circ (\mu \ot id_A)$. 
We denote by $A\ot _{R, \sigma }V$ this algebra structure and 
call it the {\em crossed product} afforded by the data $(A, V, R, \sigma )$.  
\end{proposition}

If  $A\ot _{R, \sigma }V$ is a crossed product, we introduce the 
following Sweedler-type notation:
\begin{eqnarray*}
&&R:V\ot A\rightarrow A\ot V, \;\;\;R(v\ot a)=a_R\ot v_R, \\
&&\sigma :V\ot V\rightarrow A\ot V, \;\;\;\sigma (v, v')=\sigma _1(v, v') 
\ot \sigma _2(v, v'), 
\end{eqnarray*} 
for all $v, v'\in V$ and $a\in A$. With this notation, the multiplication of 
 $A\ot _{R, \sigma }V$ reads
\begin{eqnarray*}
&&(a\ot v)(a'\ot v')=aa'_R\sigma _1(v_R, v')\ot \sigma _2(v_R, v'), \;\;\;
\forall \;a, a'\in A, \;v, v'\in V.
\end{eqnarray*}

A twisted tensor product is a particular case of a crossed product 
(cf. \cite{guccione}), namely, if $A\ot _RB$ is a twisted tensor product of 
algebras then $A\ot _RB=A\ot _{R, \sigma }B$, where 
$\sigma :B\ot B\rightarrow A\ot B$ 
is given by $\sigma (b, b')=1_A\ot bb'$, for all $b, b'\in B$. 

We recall from \cite{panaite} the invariance under twisting for crossed products:
\begin{theorem} \label{main}
Let $A\ot _{R, \sigma }V$ be a crossed product and assume 
we are given two linear maps $\theta , \gamma :V\rightarrow A\ot V$, 
with notation $\theta (v)=v_{<-1>}\ot v_{<0>}$ and 
$\gamma (v)=v_{\{-1\}}\ot v_{\{0\}}$, for all $v\in V$. Define the maps 
$R':V\ot A\rightarrow A\ot V$ and $\sigma ':V\ot V\rightarrow A\ot V$ by 
the formulae
\begin{eqnarray}
&&R'=(\mu _2\ot id_V)\circ (id_A\ot id_A\ot \gamma )\circ (id_A\ot R)\circ 
(\theta \ot id_A),   \label{Rprim}\\
&&\sigma '=(\mu \ot id_V)\circ (id_A\ot \gamma )\circ (\mu _2\ot id_V)
\circ (id_A\ot id_A\ot \sigma ) \nonumber \\
&&\;\;\;\;\;\;\;\;\;\;\;\;\circ (id_A\ot R\ot id_V)\circ 
(\theta \ot \theta  ).  \label{sigmaprim}
\end{eqnarray}
Assume that the following conditions are satisfied:
\begin{eqnarray}
&&\theta (1_V)=1_A\ot 1_V, \;\;\;\gamma (1_V)=1_A\ot 1_V, \label{cros1}\\
&&v_{<-1>}v_{<0>_{\{-1\}}}\ot v_{<0>_{\{0\}}}=1_A\ot v, 
\;\;\;\forall \;v\in V, \label{cros2}\\
&&v_{\{-1\}}v_{\{0\}_{<-1>}}\ot v_{\{0\}_{<0>}}=1_A\ot v, 
\;\;\;\forall \;v\in V, \label{cros3}\\
&&(\mu \ot id_V)\circ (\mu \ot \sigma ')\circ (id_A\ot \gamma \ot id_V)
\circ (R\ot id_V)\circ (id_V\ot \gamma )\nonumber \\
&&\;\;\;\;\;\;\;\;\;\;=(\mu \ot id_V)\circ (id_A\ot \gamma )\circ \sigma . 
\label{cros4}
\end{eqnarray}
Then  $A\ot _{R', \sigma '}V$ is a crossed product and we have an algebra 
isomorphism  $\varphi :A\ot _{R', \sigma '}V\simeq A\ot _{R, \sigma }V$, 
$\varphi (a\ot v)=av_{<-1>}\ot v_{<0>}$, for all $a\in A$, $v\in V$.  
\end{theorem}
\section{The characterization of equivalent crossed products}
\setcounter{equation}{0}
\begin{definition}
Let $(A, \mu , 1_A)$ be an (associative unital) algebra, $V$ a 
vector space equipped with a distinguished element $1_V\in V$ and 
$A\ot _{R, \sigma }V$, $A\ot _{R', \sigma '}V$ two crossed products. 
We say that $A\ot _{R, \sigma }V$ and $A\ot _{R', \sigma '}V$ are 
{\bf equivalent} if there exists an algebra isomorphism 
$\varphi :A\ot _{R', \sigma '}V\simeq A\ot _{R, \sigma }V$ with 
the property that $\varphi (a\ot 1_V)=a\ot 1_V$ for all $a\in A$. 
\end{definition}
\begin{remark}
A crossed product algebra $A\ot _{R, \sigma }V$ becomes canonically 
a left $A$-module, by 
\begin{eqnarray*}
&&A\ot (A\ot _{R, \sigma }V)\rightarrow 
A\ot _{R, \sigma }V, \;\;\; a'\ot (a\ot v)\mapsto a'a\ot v.
\end{eqnarray*}
Then it is easy to see that two crossed products 
$A\ot _{R, \sigma }V$ and $A\ot _{R', \sigma '}V$ are 
equivalent if and only if there exists a linear isomorphism 
$\varphi :A\ot _{R', \sigma '}V\simeq A\ot _{R, \sigma }V$ 
which is an algebra map and a morphism of left $A$-modules. 
\end{remark}
\begin{theorem} \label{mainconverse}
Let $A\ot _{R, \sigma }V$ and $A\ot _{R', \sigma '}V$ be two crossed 
products. Then $A\ot _{R, \sigma }V$ and $A\ot _{R', \sigma '}V$ are 
equivalent if and only if there exist linear maps 
$\theta , \gamma :V\rightarrow A\ot V$, 
with notation $\theta (v)=v_{<-1>}\ot v_{<0>}$ and 
$\gamma (v)=v_{\{-1\}}\ot v_{\{0\}}$, for all $v\in V$, such that 
the conditions (\ref{Rprim})--(\ref{cros4}) are satisfied. 
\end{theorem}
\begin{proof}
One implication is exactly Theorem \ref{main}, so we only have to prove 
the converse. Assume that there exists an algebra isomorphism 
$\varphi :A\ot _{R', \sigma '}V\simeq A\ot _{R, \sigma }V$ such that  
$\varphi (a\ot 1_V)=a\ot 1_V$ for all $a\in A$, with inverse denoted by 
$\varphi ^{-1}:A\ot _{R, \sigma }V\simeq A\ot _{R', \sigma '}V$. 
Define the maps
\begin{eqnarray}
&&\theta :V\rightarrow A\ot V, \;\;\;\theta (v)=\varphi (1_A\ot v):=
v_{<-1>}\ot v_{<0>}, \\
&&\gamma :V\rightarrow A\ot V, \;\;\;\gamma (v)=\varphi ^{-1}(1_A\ot v):=
v_{\{-1\}}\ot v_{\{0\}}. 
\end{eqnarray}
Since $\varphi $ is an algebra isomorphism, for all $a\in A$, $v\in V$ we have 
$\varphi (a\ot v)=\varphi ((a\ot 1_V)(1_A\ot v))=\varphi (a\ot 1_V)
\varphi (1_A\ot v)=(a\ot 1_V)(v_{<-1>}\ot v_{<0>})=
av_{<-1>}\ot v_{<0>}$ and similarly $\varphi ^{-1}(a\ot v)=
av_{\{-1\}}\ot v_{\{0\}}$. Since $\varphi (1_A\ot 1_V)=
\varphi ^{-1}(1_A\ot 1_V)=1_A\ot 1_V$ we obtain 
$\theta (1_V)=\gamma (1_V)=1_A\ot 1_V$, that is (\ref{cros1}) holds. 
Now we compute, for $v\in V$:
\begin{eqnarray*}
&&1_A\ot v=\varphi ^{-1}(\varphi (1_A\ot v))=
\varphi ^{-1}(v_{<-1>}\ot v_{<0>})=
v_{<-1>}v_{<0>_{\{-1\}}}\ot v_{<0>_{\{0\}}}, \\
&&1_A\ot v=\varphi (\varphi ^{-1}(1_A\ot v))=
\varphi (v_{\{-1\}}\ot v_{\{0\}})=
v_{\{-1\}}v_{\{0\}_{<-1>}}\ot v_{\{0\}_{<0>}}, 
\end{eqnarray*}
that is (\ref{cros2}) and (\ref{cros3}) hold. Let now $v\in V$, $a\in A$; 
we compute:
\begin{eqnarray*}
R'(v\ot a)&=&\varphi ^{-1}(\varphi (R'(v\ot a)))\\
&=&\varphi ^{-1}(\varphi ((1_A\ot v)(a\ot 1_V)))\\
&=&\varphi ^{-1}(\varphi (1_A\ot v)\varphi (a\ot 1_V))\\
&=&\varphi ^{-1}((v_{<-1>}\ot v_{<0>})(a\ot 1_V))\\
&=&\varphi ^{-1}(v_{<-1>}a_R\ot v_{<0>_R})\\
&=&v_{<-1>}a_Rv_{<0>_{R_{\{-1\}}}}\ot 
v_{<0>_{R_{\{0\}}}},
\end{eqnarray*}
and this is exactly the relation (\ref{Rprim}).  Let now 
$v, w\in V$; we compute:
\begin{eqnarray*}
\sigma '(v, w)&=&\varphi ^{-1}(\varphi (\sigma '(v, w)))\\
&=&\varphi ^{-1}(\varphi ((1_A\ot v)(1_A\ot w)))\\
&=&\varphi ^{-1}(\varphi (1_A\ot v)\varphi (1_A\ot w))\\
&=&\varphi ^{-1}((v_{<-1>}\ot v_{<0>})(w_{<-1>}\ot w_{<0>}))\\
&=&\varphi ^{-1}(v_{<-1>}w_{<-1>_R}\sigma _1(v_{<0>_R}, 
w_{<0>})\ot \sigma _2(v_{<0>_R}, w_{<0>}))\\
&=&v_{<-1>}w_{<-1>_R}\sigma _1(v_{<0>_R}, 
w_{<0>})\sigma _2(v_{<0>_R}, w_{<0>})_{\{-1\}}\ot 
\sigma _2(v_{<0>_R}, w_{<0>})_{\{0\}},
\end{eqnarray*}
and this is exactly the relation (\ref{sigmaprim}). 

The only thing left to prove is the relation (\ref{cros4}). In order to 
show that the left hand side and the right hand side in 
(\ref{cros4}) are equal, it is enough to prove that they are 
equal after composing with $\varphi $. Thus, for 
$v, w\in V$, we compute (denoting by $r$ another copy of $R$):\\[2mm]
${\;\;\;\;\;}$$\varphi \circ (\mu \ot id_V)\circ (\mu \ot \sigma ')\circ 
(id_A\ot \gamma \ot id_V)\circ (R\ot id_V)\circ (id_V\ot \gamma )(v\ot w)$
\begin{eqnarray*}
&=&\varphi \circ (\mu \ot id_V)\circ (\mu \ot \sigma ')\circ 
(id_A\ot \gamma \ot id_V)\circ (R\ot id_V)(v\ot w_{\{-1\}}\ot w_{\{0\}})\\
&=&\varphi \circ (\mu \ot id_V)\circ (\mu \ot \sigma ')\circ 
(id_A\ot \gamma \ot id_V)(w_{\{-1\}_R}\ot v_R\ot w_{\{0\}})\\
&=&\varphi \circ (\mu \ot id_V)\circ (\mu \ot \sigma ')
(w_{\{-1\}_R}\ot v_{R_{\{-1\}}}\ot  v_{R_{\{0\}}}\ot 
w_{\{0\}})\\
&\overset{(\ref{sigmaprim})}{=}&
\varphi \circ (\mu \ot id_V)(w_{\{-1\}_R}v_{R_{\{-1\}}}\ot 
v_{R_{\{0\}_{<-1>}}}w_{\{0\}_{<-1>_r}}
\sigma _1(v_{R_{\{0\}_{<0>_r}}}, w_{\{0\}_{<0>}})\\
&&\sigma _2(v_{R_{\{0\}_{<0>_r}}}, w_{\{0\}_{<0>}})_{\{-1\}}
\ot \sigma _2(v_{R_{\{0\}_{<0>_r}}}, w_{\{0\}_{<0>}})_{\{0\}})\\
&=&\varphi (w_{\{-1\}_R}v_{R_{\{-1\}}}
v_{R_{\{0\}_{<-1>}}}w_{\{0\}_{<-1>_r}}
\sigma _1(v_{R_{\{0\}_{<0>_r}}}, w_{\{0\}_{<0>}})\\
&&\sigma _2(v_{R_{\{0\}_{<0>_r}}}, w_{\{0\}_{<0>}})_{\{-1\}}
\ot \sigma _2(v_{R_{\{0\}_{<0>_r}}}, w_{\{0\}_{<0>}})_{\{0\}})\\
&=&w_{\{-1\}_R}v_{R_{\{-1\}}}
v_{R_{\{0\}_{<-1>}}}w_{\{0\}_{<-1>_r}}
\sigma _1(v_{R_{\{0\}_{<0>_r}}}, w_{\{0\}_{<0>}})\\
&&\sigma _2(v_{R_{\{0\}_{<0>_r}}}, w_{\{0\}_{<0>}})_{\{-1\}}
\sigma _2(v_{R_{\{0\}_{<0>_r}}}, w_{\{0\}_{<0>}})_{\{0\}_{<-1>}}\\
&&\ot 
\sigma _2(v_{R_{\{0\}_{<0>_r}}}, w_{\{0\}_{<0>}})_{\{0\}_{<0>}}\\
&\overset{(\ref{cros3})}{=}&
w_{\{-1\}_R}v_{R_{\{-1\}}}
v_{R_{\{0\}_{<-1>}}}w_{\{0\}_{<-1>_r}}
\sigma _1(v_{R_{\{0\}_{<0>_r}}}, w_{\{0\}_{<0>}})\\
&&\ot \sigma _2(v_{R_{\{0\}_{<0>_r}}}, w_{\{0\}_{<0>}})\\
&\overset{(\ref{cros3})}{=}&
w_{\{-1\}_R}w_{\{0\}_{<-1>_r}}
\sigma _1(v_{R_r}, w_{\{0\}_{<0>}})
\ot \sigma _2(v_{R_r}, w_{\{0\}_{<0>}})\\
&\overset{(\ref{brz3})}{=}&
(w_{\{-1\}}w_{\{0\}_{<-1>}})_R
\sigma _1(v_R, w_{\{0\}_{<0>}})
\ot \sigma _2(v_R, w_{\{0\}_{<0>}})\\
&\overset{(\ref{cros3})}{=}&
(1_A)_R\sigma _1(v_R, w)\ot \sigma _2(v_R, w)\\
&\overset{(\ref{brz1})}{=}&\sigma (v\ot w), 
\end{eqnarray*}
${\;\;\;\;\;}$$\varphi \circ (\mu \ot id_V)\circ (id_A\ot \gamma )\circ \sigma (v\ot w)$
\begin{eqnarray*}
&=&\varphi \circ (\mu \ot id_V)\circ (id_A\ot \gamma )(\sigma _1(v, w)
\ot \sigma _2(v, w))\\
&=&\varphi \circ (\mu \ot id_V)(\sigma _1(v, w)\ot \sigma _2(v, w)_{\{-1\}}
\ot \sigma _2(v, w)_{\{0\}})\\
&=&\varphi (\sigma _1(v, w)\sigma _2(v, w)_{\{-1\}}
\ot \sigma _2(v, w)_{\{0\}})\\
&=&\sigma _1(v, w)\sigma _2(v, w)_{\{-1\}}\sigma _2(v, w)_{\{0\}_{<-1>}}
\ot \sigma _2(v, w)_{\{0\}_{<0>}}\\
&\overset{(\ref{cros3})}{=}&\sigma _1(v, w)\ot \sigma _2(v, w)\\
&=&\sigma (v\ot w), 
\end{eqnarray*}
and we can see that the two terms are equal. 
\end{proof}

With a similar proof, we can obtain a converse of the invariance under 
twisting for twisted tensor products of algebras from \cite{jlpvo}:
\begin{definition}
Let $A$ be an associative unital algebra and $B$, $B'$ two associative unital 
algebra structures with the same unit $1_B$ on the vector space $B$. 
If $A\ot _RB$ and $A\ot _{R'}B'$ are two twisted tensor products of 
algebras, 
we call them $A$-equivalent if there exists an algebra isomorphism 
$\varphi :A\ot _{R'}B'\simeq A\ot _RB$ such that $\varphi (a\ot 1_B)=
a\ot 1_B$ for all $a\in A$. 
\end{definition}
\begin{theorem}
Let $A$ be an associative unital algebra and $B$, $B'$ two associative unital 
algebra structures with the same unit $1_B$ on the vector space $B$. 
Denote the multiplication of $B$ by $b\ot b'\mapsto bb'$ and the 
multiplication of $B'$ by $b\ot b'\mapsto b*b'$. 
If $A\ot _RB$ and $A\ot _{R'}B'$ are two twisted tensor products 
of algebras, then they are $A$-equivalent if and only if there exist 
two linear maps 
$\theta ,\gamma :B\rightarrow A\ot B$, with notation $\theta
(b)=b_{<-1>}\ot b_{<0>}$ and $\gamma (b)=b_{\{-1\}}\ot b_{\{0\}}$,
such that $\theta $ is an algebra map from $B'$ to $A\ot _RB$,
$\gamma (1_B)=1_A\ot 1_B$ and for all $a\in A$ and $b, b'\in B$ we have:
\begin{eqnarray}
&&\gamma (bb')=b'_{\{-1\}_R}b_{R_{\{-1\}}}\ot
b_{R_{\{0\}}}*b'_{\{0\}},
\label{rel1} \\
&&b_{<-1>}b_{<0>_{\{-1\}}}\ot b_{<0>_{\{0\}}}=1\ot b, \label{rel2}\\
&&b_{\{-1\}}b_{\{0\}_{<-1>}}\ot b_{\{0\}_{<0>}}=1\ot b,  \label{rel3}\\
&&R'(b\ot a)=b_{<-1>}a_R
b_{<0>_{R_{\{-1\}}}}\ot b_{<0>_{R_{\{0\}}}}. \label{rel4}
\end{eqnarray}
\end{theorem}
\begin{remark}
The same type of result as in Theorem \ref{mainconverse} was proved in 
\cite{am1} for bicrossed products of Hopf algebras. 
\end{remark}
\section{Equivalent cross product bialgebras}
\setcounter{equation}{0}
${\;\;\;\;}$
We begin by recalling the dual concept of Brzezi\'{n}ski's crossed products, 
which appears (in a slightly different form) in \cite{BDII}:
\begin{proposition}
Let $(C, \Delta _C, \varepsilon _C)$ be a (coassociative counital) 
coalgebra and $X$ a vector space endowed with a linear map 
$\varepsilon _X:X\rightarrow k$. Then the vector space 
$X\ot C$ is a coassociative coalgebra with counit 
$\varepsilon _X\ot \varepsilon _C$ and whose comultiplication  
$\Delta $ has the property that 
$(id_X\ot id_C\ot \varepsilon _X\ot id_C)\circ \Delta =id_X\ot \Delta _C$ 
if and only if there exist linear maps $W:X\ot C\rightarrow C\ot X$ and 
$\rho :X\ot C\rightarrow X\ot X$ satisfying the following conditions:
\begin{eqnarray}
&&(id_C\ot \varepsilon _X)\circ W=\varepsilon _X\ot id_C, \;\;\;
(\varepsilon _C\ot id_X)\circ W=id_X\ot \varepsilon _C, 
\label{cobrz1}\\
&&(id_X\ot \varepsilon _X)\circ \rho =id_X\ot \varepsilon _C, \;\;\;
(\varepsilon _X\ot id_X)\circ \rho =id_X\ot \varepsilon _C, 
\label{cobrz2}\\
&&(\Delta _C\ot id_X)\circ W=(id_C\ot W)\circ (W\ot id_C)
\circ (id_X\ot \Delta _C), \label{cobrz3} \\ 
&&(\rho \ot id_X)\circ (id_X\ot W)\circ (\rho \ot id_C)\circ 
(id_X\ot \Delta _C)\nonumber \\
&&\;\;\;\;\;\;\;\;\;\;
=(id_X\ot \rho )\circ (\rho \ot id_C)\circ 
(id_X\ot \Delta _C), \label{cobrz4} \\
&&(W\ot id_X)\circ (id_X\ot W)\circ (\rho \ot id_C)\circ 
(id_X\ot \Delta _C)\nonumber \\
&&\;\;\;\;\;\;\;\;\;\;
=(id_C\ot \rho )\circ (W\ot id_C)\circ (id_X\ot \Delta _C). 
\label{cobrz5}
\end{eqnarray}
In this case, the comultiplication of $X\ot C$  is given 
explicitely by 
\begin{eqnarray}
&&\Delta =(id_X\ot W\ot id_C)\circ (\rho \ot \Delta _C)
\circ (id_X\ot \Delta _C). \label{formco}
\end{eqnarray}
We denote by $X_{\;W, \rho }\ot C$ this coalgebra 
structure and call it the {\em crossed coproduct} afforded 
by the data $(X, C, W, \rho )$. 
\end{proposition}

If $X_{\;W, \rho }\ot C$ is a crossed coproduct, we 
introduce the following Sweedler-type notation: 
\begin{eqnarray*}
&&W:X\ot C\rightarrow C\ot X, \;\;\;W(x\ot c)=c_W\ot x_W, \\
&&\rho :X\ot C\rightarrow X\ot X, \;\;\;\rho (x\ot c)=
\rho _1(x, c)\ot \rho _2(x, c), 
\end{eqnarray*}
for all $x\in X$, $c\in C$. With this notation, the 
comultiplication of $X_{\;W, \rho }\ot C$ reads: 
\begin{eqnarray*}
&&\Delta (x\ot c)=(\rho _1(x, c_1)\ot c_{2_W})\ot 
(\rho _2(x, c_1)_W\ot c_3), \;\;\;\forall \;x\in X, c\in C. 
\end{eqnarray*}
\begin{remark} \label{tare}
It is easy to see that, if $X_{\;W, \rho }\ot C$ is a crossed coproduct 
as above, then 
\begin{eqnarray*}
&&W=(\varepsilon _X\ot id_C\ot id_X\ot \varepsilon _C)
\circ \Delta , \\
&&\rho =(id_X\ot \varepsilon _C\ot id_X\ot \varepsilon _C)
\circ \Delta ,
\end{eqnarray*}
where $\Delta $ is the comultiplication of 
$X_{\;W, \rho }\ot C$.
\end{remark}

We recall the following concept from \cite{BDII}:
\begin{definition} (\cite{BDII})
A bialgebra $B$ is called a {\em cross product bialgebra} 
if its underlying algebra structure is a crossed product algebra 
$A\ot _{R, \sigma }C$ and its underlying coalgebra structure 
is a crossed coproduct coalgebra $A_{\;W, \rho }\ot C$ 
on the same objects. This cross product bialgebra $B$ 
will be denoted by 
$A _{\;W}^{\;\rho}\bowtie _{\;R}^{\;\sigma }C$. 
\end{definition}
\begin{remark}
If $A _{\;W}^{\;\rho}\bowtie _{\;R}^{\;\sigma }C$ is a cross 
product bialgebra, then the vector space $A\ot C$ 
becomes canonically a left $A$-module and a right $C$-comodule 
as follows: 
\begin{eqnarray*}
&&A\ot (A\ot C)\rightarrow 
A\ot C, \;\;\; a'\ot (a\ot c)\mapsto a'a\ot c, \\
&&A\ot C\rightarrow (A\ot C)\ot C, \;\;\;a\ot c\mapsto 
(a\ot c_1)\ot c_2,
\end{eqnarray*}
for all $a, a'\in A$ and $c\in C$. 
\end{remark}
\begin{definition}
Two cross product bialgebras 
$A _{\;W}^{\;\rho}\bowtie _{\;R}^{\;\sigma }C$ and 
$A _{\;W'}^{\;\rho '}\bowtie _{\;R'}^{\;\sigma '}C$ 
will be called {\em equivalent} if there exists a 
linear isomorphism 
$\varphi :A _{\;W'}^{\;\rho '}\bowtie _{\;R'}^{\;\sigma '}C
\simeq A _{\;W}^{\;\rho}\bowtie _{\;R}^{\;\sigma }C$ which is a 
morphism of bialgebras, of left $A$-modules and of right 
$C$-comodules. 
\end{definition}

Let now $(A, \mu _A, 1_A, \Delta _A, \varepsilon _A)$ be a 
bialgebra and $(C, \Delta _C, \varepsilon _C)$ a coalgebra. 
Assume that we have a crossed product algebra 
$A\ot _{R, \sigma }C$. Define the maps 
\begin{eqnarray*}
&&W_0:A\ot C\rightarrow C\ot A, \;\;\;W_0(a\ot c)=c\ot a, \\
&&\rho _0:A\ot C\rightarrow A\ot A, \;\;\;\rho _0(a\ot c)=
a_1\ot a_2\varepsilon _C(c).
\end{eqnarray*}
Then it is easy to see that  $A_{\;W_0, \rho _0}\ot C$ is 
a crossed coproduct coalgebra, which is actually 
the tensor product coalgebra $A\ot C$. 

Assume that moreover 
$A _{\;W_0}^{\;\rho _0}\bowtie _{\;R}^{\;\sigma }C$ is a 
cross product bialgebra, that is the maps 
\begin{eqnarray*}
&&\Delta :A\ot _{R, \sigma }C\rightarrow 
(A\ot _{R, \sigma }C)\ot (A\ot _{R, \sigma }C), \;\;\;
\Delta (a\ot c)=(a_1\ot c_1)\ot (a_2\ot c_2), \\
&&\varepsilon _A\ot \varepsilon _C:A\ot _{R, \sigma }C
\rightarrow k,  
\end{eqnarray*}
are algebra maps. Our aim is to describe all cross product 
bialgebras that are equivalent to 
$A _{\;W_0}^{\;\rho _0}\bowtie _{\;R}^{\;\sigma }C$.
\begin{theorem}\label{maincros}
In the above hypotheses, a cross product bialgebra 
$A _{\;W'}^{\;\rho '}\bowtie _{\;R'}^{\;\sigma '}C$ is 
equivalent to $A _{\;W_0}^{\;\rho _0}\bowtie _{\;R}^{\;\sigma }C$ 
if and only if there exist linear maps 
$\theta , \gamma :C\rightarrow A\ot C$, 
with notation $\theta (c)=c_{<-1>}\ot c_{<0>}$ and 
$\gamma (c)=c_{\{-1\}}\ot c_{\{0\}}$, for all $c\in C$, 
such that the conditions (\ref{Rprim})--(\ref{cros4}) are 
satisfied and moreover we have, 
for all $a\in A$, $c\in C$:
\begin{eqnarray}
&&W'(a\ot c)=\varepsilon _A(c_{<0>_{1_{\{-1\}}}})
\varepsilon _C(c_{<0>_{2_{\{0\}}}})c_{<0>_{1_{\{0\}}}}
\ot ac_{<-1>}c_{<0>_{2_{\{-1\}}}}, \label{Wprim}\\
&&\rho '(a\ot c)=\varepsilon _C(c_{<0>_{\{0\}_{\{0\}}}})
a_1c_{<-1>_1}c_{<0>_{\{-1\}}}\ot a_2c_{<-1>_2}
c_{<0>_{\{0\}_{\{-1\}}}}, \label{rhoprim}\\
&&\varepsilon _A(c_{<-1>})\varepsilon _C(c_{<0>})
=\varepsilon _C(c)=
\varepsilon _A(c_{\{-1\}})\varepsilon _C(c_{\{0\}}), 
\label{extra1}\\
&&c_{<-1>}\ot c_{<0>_1}\ot c_{<0>_2}=
c_{1_{<-1>}}\ot c_{1_{<0>}}\ot c_2, \label{extra2} \\
&&c_{\{-1\}}\ot c_{\{0\}_1}\ot c_{\{0\}_2}=
c_{1_{\{-1\}}}\ot c_{1_{\{0\}}}\ot c_2. \label{extra3}
\end{eqnarray} 
\end{theorem}
\begin{proof}
We prove first that if all those conditions are satisfied then 
$A _{\;W'}^{\;\rho '}\bowtie _{\;R'}^{\;\sigma '}C$ is 
indeed a cross product bialgebra and moreover it is 
equivalent to $A _{\;W_0}^{\;\rho _0}\bowtie _{\;R}^{\;\sigma }C$. 
We begin by obtaining some important consequences of the relations. 
We compute: 
\begin{eqnarray*}
\varepsilon _C(c_{1_{<0>}})c_{1_{<-1>}}c_{2_{\{-1\}}}\ot 
c_{2_{\{0\}}}
&\overset{(\ref{extra2})}{=}&
\varepsilon _C(c_{<0>_1})c_{<-1>}c_{<0>_{2_{\{-1\}}}}\ot 
c_{<0>_{2_{\{0\}}}}\\
&=&c_{<-1>}c_{<0>_{\{-1\}}}\ot c_{<0>_{\{0\}}}\\
&\overset{(\ref{cros2})}{=}&1_A\ot c.
\end{eqnarray*}
We apply $\varepsilon _A\ot id_C$ to this relation and we obtain: 
\begin{eqnarray*}
c&=&\varepsilon _C(c_{1_{<0>}})\varepsilon _A(c_{1_{<-1>}})
\varepsilon _A(c_{2_{\{-1\}}})c_{2_{\{0\}}}\\
&\overset{(\ref{extra1})}{=}&
\varepsilon _C(c_1)
\varepsilon _A(c_{2_{\{-1\}}})c_{2_{\{0\}}}\\
&=&\varepsilon _A(c_{\{-1\}})c_{\{0\}}.
\end{eqnarray*}
Also, for $c\in C$, we compute:
\begin{eqnarray*}
\varepsilon _C(c_{1_{\{0\}}})\varepsilon _C(c_{2_{\{0\}}})
c_{1_{\{-1\}}}\ot c_{2_{\{-1\}}}
&\overset{(\ref{extra3})}{=}&
\varepsilon _C(c_{\{0\}_1})\varepsilon _C(c_{\{0\}_{2_{\{0\}}}})
c_{\{-1\}}\ot c_{\{0\}_{2_{\{-1\}}}}\\
&=&\varepsilon _C(c_{\{0\}_{\{0\}}})
c_{\{-1\}}\ot c_{\{0\}_{\{-1\}}}.
\end{eqnarray*}
Let us record for future use these three relations and the other three 
obtained by interchanging $<>$ and $\{\}$: 
\begin{eqnarray}
&&\varepsilon _C(c_{1_{<0>}})c_{1_{<-1>}}c_{2_{\{-1\}}}\ot 
c_{2_{\{0\}}}=1_A\ot c, \label{cucu1}\\
&&\varepsilon _C(c_{1_{\{0\}}})c_{1_{\{-1\}}}c_{2_{<-1>}}\ot 
c_{2_{<0>}}=1_A\ot c, \label{cucu2}\\
&&\varepsilon _A(c_{<-1>})c_{<0>}=c, \label{cucu3}\\
&&\varepsilon _A(c_{\{-1\}})c_{\{0\}}=c, \label{cucu4} \\
&&\varepsilon _C(c_{1_{\{0\}}})\varepsilon _C(c_{2_{\{0\}}})
c_{1_{\{-1\}}}\ot c_{2_{\{-1\}}}=
\varepsilon _C(c_{\{0\}_{\{0\}}})
c_{\{-1\}}\ot c_{\{0\}_{\{-1\}}}, \label{cucu5}\\
&&\varepsilon _C(c_{1_{<0>}})\varepsilon _C(c_{2_{<0>}})
c_{1_{<-1>}}\ot c_{2_{<-1>}}=
\varepsilon _C(c_{<0>_{<0>}})
c_{<-1>}\ot c_{<0>_{<-1>}}, \label{cucu6}
\end{eqnarray}
for all $c\in C$. Note that by using (\ref{cucu4}), 
the formula for $W'$ may be written as 
\begin{eqnarray}
&&W'(a\ot c)=\varepsilon _C(c_{<0>_{2_{\{0\}}}})
c_{<0>_1}\ot ac_{<-1>}c_{<0>_{2_{\{-1\}}}}. \label{Wprimnou}
\end{eqnarray}

We prove now that for $W'$ and $\rho '$ defined by 
(\ref{Wprim}) and respectively (\ref{rhoprim}), 
$A_{\;W', \rho '}\ot C$ is a crossed coproduct coalgebra. The 
relations (\ref{cobrz1}) and (\ref{cobrz2}) for 
$W'$ and $\rho '$ are easy to prove and are left to the reader. \\[2mm]
\underline{Proof of (\ref{cobrz3})}: We compute, for $a\in A$ and 
$c\in C$:
\begin{eqnarray*}
(\Delta _C\ot id_A)(W'(a\ot c))&=&
\varepsilon _C(c_{<0>_{3_{\{0\}}}})c_{<0>_1}\ot 
c_{<0>_2}\ot ac_{<-1>}c_{<0>_{3_{\{-1\}}}}\\
&\overset{(\ref{extra2})}{=}&
\varepsilon _C(c_{3_{\{0\}}})c_{1_{<0>}}\ot 
c_2\ot ac_{1_{<-1>}}c_{3_{\{-1\}}},
\end{eqnarray*}
${\;\;\;\;}$$(id_C\ot W')\circ (W'\ot id_C)\circ (id_A\ot \Delta _C)(a\ot c)$
\begin{eqnarray*}
&=&(id_C\ot W')\circ (W'\ot id_C)(a\ot c_1\ot c_2)\\
&=&(id_C\ot W')(\varepsilon _C(c_{1_{<0>_{2_{\{0\}}}}})c_{1_{<0>_1}}\ot 
ac_{1_{<-1>}}c_{1_{<0>_{2_{\{-1\}}}}}\ot c_2)\\
&=&\varepsilon _C(c_{1_{<0>_{2_{\{0\}}}}})c_{1_{<0>_1}}\ot 
\varepsilon _C(c_{2_{<0>_{2_{\{0\}}}}})c_{2_{<0>_1}}\ot 
ac_{1_{<-1>}}c_{1_{<0>_{2_{\{-1\}}}}}c_{2_{<-1>}}
c_{2_{<0>_{2_{\{-1\}}}}}\\
&\overset{(\ref{extra2})}{=}&
\varepsilon _C(c_{1_{<0>_{2_{\{0\}}}}})c_{1_{<0>_1}}\ot 
\varepsilon _C(c_{2_{2_{\{0\}}}})c_{2_{1_{<0>}}}\ot 
ac_{1_{<-1>}}c_{1_{<0>_{2_{\{-1\}}}}}c_{2_{1_{<-1>}}}
c_{2_{2_{\{-1\}}}}\\
&=&\varepsilon _C(c_{1_{<0>_{2_{\{0\}}}}})c_{1_{<0>_1}}\ot 
\varepsilon _C(c_{3_{\{0\}}})c_{2_{<0>}}\ot 
ac_{1_{<-1>}}c_{1_{<0>_{2_{\{-1\}}}}}c_{2_{<-1>}}
c_{3_{\{-1\}}}\\
&\overset{(\ref{extra2})}{=}&
\varepsilon _C(c_{2_{\{0\}}})c_{1_{<0>}}\ot 
\varepsilon _C(c_{4_{\{0\}}})c_{3_{<0>}}\ot 
ac_{1_{<-1>}}c_{2_{\{-1\}}}c_{3_{<-1>}}
c_{4_{\{-1\}}}\\
&\overset{(\ref{cucu2})}{=}&
\varepsilon _C(c_{3_{\{0\}}})c_{1_{<0>}}\ot 
c_2\ot ac_{1_{<-1>}}c_{3_{\{-1\}}}, \;\;\;q.e.d.
\end{eqnarray*}
\underline{Proof of (\ref{cobrz4})}: We compute, for $a\in A$ and 
$c\in C$:\\[2mm]
${\;\;}$$(\rho '\ot id_A)\circ (id_A\ot W')\circ (\rho '\ot id_C)\circ 
(id_A\ot \Delta _C)(a\ot c)$
\begin{eqnarray*}
&=&(\rho '\ot id_A)\circ (id_A\ot W')\circ (\rho '\ot id_C)
(a\ot c_1\ot c_2)\\
&=&(\rho '\ot id_A)\circ (id_A\ot W')
(\varepsilon _C(c_{1_{<0>_{\{0\}_{\{0\}}}}})a_1c_{1_{<-1>_1}}
c_{1_{<0>_{\{-1\}}}}\ot a_2c_{1_{<-1>_2}}
c_{1_{<0>_{\{0\}_{\{-1\}}}}}\ot c_2)\\
&=&(\rho '\ot id_A)(\varepsilon _C(c_{1_{<0>_{\{0\}_{\{0\}}}}})
a_1c_{1_{<-1>_1}}c_{1_{<0>_{\{-1\}}}}\ot 
\varepsilon _C(c_{2_{<0>_{2_{\{0\}}}}})c_{2_{<0>_1}}\\
&&\ot a_2c_{1_{<-1>_2}}c_{1_{<0>_{\{0\}_{\{-1\}}}}}
c_{2_{<-1>}}c_{2_{<0>_{2_{\{-1\}}}}})\\
&\overset{(\ref{extra2})}{=}&
(\rho '\ot id_A)(\varepsilon _C(c_{<0>_{1_{\{0\}_{\{0\}}}}})
\varepsilon _C(c_{<0>_{2_{<0>_{2_{\{0\}}}}}})
a_1c_{<-1>_1}c_{<0>_{1_{\{-1\}}}}\ot 
c_{<0>_{2_{<0>_1}}}\\
&&\ot a_2c_{<-1>_2}c_{<0>_{1_{\{0\}_{\{-1\}}}}}
c_{<0>_{2_{<-1>}}}c_{<0>_{2_{<0>_{2_{\{-1\}}}}}})\\
&\overset{(\ref{extra3})}{=}&
(\rho '\ot id_A)(\varepsilon _C(c_{<0>_{\{0\}_{1_{\{0\}}}}})
\varepsilon _C(c_{<0>_{\{0\}_{2_{<0>_{2_{\{0\}}}}}}})
a_1c_{<-1>_1}c_{<0>_{\{-1\}}}\ot 
c_{<0>_{\{0\}_{2_{<0>_1}}}}\\
&&\ot a_2c_{<-1>_2}c_{<0>_{\{0\}_{1_{\{-1\}}}}}
c_{<0>_{\{0\}_{2_{<-1>}}}}c_{<0>_{\{0\}_{2_{<0>_{2_{\{-1\}}}}}}})\\
&\overset{(\ref{cucu2})}{=}&
(\rho '\ot id_A)(\varepsilon _C(c_{<0>_{\{0\}_{2_{\{0\}}}}})
a_1c_{<-1>_1}c_{<0>_{\{-1\}}}\ot c_{<0>_{\{0\}_1}}\ot 
a_2c_{<-1>_2}c_{<0>_{\{0\}_{2_{\{-1\}}}}})\\
&\overset{(\ref{extra3})}{=}&
(\rho '\ot id_A)(\varepsilon _C(c_{<0>_{2_{\{0\}}}})
a_1c_{<-1>_1}c_{<0>_{1_{\{-1\}}}}\ot c_{<0>_{1_{\{0\}}}}\ot 
a_2c_{<-1>_2}c_{<0>_{2_{\{-1\}}}})\\
&\overset{(\ref{extra2})}{=}&
(\rho '\ot id_A)(\varepsilon _C(c_{2_{\{0\}}})
a_1c_{1_{<-1>_1}}c_{1_{<0>_{\{-1\}}}}\ot c_{1_{<0>_{\{0\}}}}\ot 
a_2c_{1_{<-1>_2}}c_{2_{\{-1\}}})\\
&=&\varepsilon _C(c_{2_{\{0\}}})\varepsilon _C
(c_{1_{<0>_{\{0\}_{<0>_{\{0\}_{\{0\}}}}}}})
a_1c_{1_{<-1>_1}}c_{1_{<0>_{\{-1\}_1}}}
c_{1_{<0>_{\{0\}_{<-1>_1}}}}
c_{1_{<0>_{\{0\}_{<0>_{\{-1\}}}}}}\\
&&\ot 
a_2c_{1_{<-1>_2}}c_{1_{<0>_{\{-1\}_2}}}
c_{1_{<0>_{\{0\}_{<-1>_2}}}}
c_{1_{<0>_{\{0\}_{<0>_{\{0\}_{\{-1\}}}}}}}\ot 
a_3c_{1_{<-1>_3}}c_{2_{\{-1\}}}\\
&=&\varepsilon _C(c_{2_{\{0\}}})\varepsilon _C
(c_{1_{<0>_{\{0\}_{<0>_{\{0\}_{\{0\}}}}}}})
a_1c_{1_{<-1>_1}}(c_{1_{<0>_{\{-1\}}}}
c_{1_{<0>_{\{0\}_{<-1>}}}})_1
c_{1_{<0>_{\{0\}_{<0>_{\{-1\}}}}}}\\
&&\ot 
a_2c_{1_{<-1>_2}}(c_{1_{<0>_{\{-1\}}}}
c_{1_{<0>_{\{0\}_{<-1>}}}})_2
c_{1_{<0>_{\{0\}_{<0>_{\{0\}_{\{-1\}}}}}}}\ot 
a_3c_{1_{<-1>_3}}c_{2_{\{-1\}}}\\
&\overset{(\ref{cros3})}{=}&
\varepsilon _C(c_{2_{\{0\}}})\varepsilon _C
(c_{1_{<0>_{\{0\}_{\{0\}}}}})
a_1c_{1_{<-1>_1}}
c_{1_{<0>_{\{-1\}}}}\ot 
a_2c_{1_{<-1>_2}}
c_{1_{<0>_{\{0\}_{\{-1\}}}}}\ot 
a_3c_{1_{<-1>_3}}c_{2_{\{-1\}}}\\
&\overset{(\ref{extra2})}{=}&
\varepsilon _C(c_{<0>_{2_{\{0\}}}})\varepsilon _C
(c_{<0>_{1_{\{0\}_{\{0\}}}}})
a_1c_{<-1>_1}
c_{<0>_{1_{\{-1\}}}}\ot 
a_2c_{<-1>_2}
c_{<0>_{1_{\{0\}_{\{-1\}}}}}\\
&&\ot 
a_3c_{<-1>_3}c_{<0>_{2_{\{-1\}}}}\\
&\overset{(\ref{extra3})}{=}&
\varepsilon _C(c_{<0>_{\{0\}_{2_{\{0\}}}}})\varepsilon _C
(c_{<0>_{\{0\}_{1_{\{0\}}}}})
a_1c_{<-1>_1}
c_{<0>_{\{-1\}}}\ot 
a_2c_{<-1>_2}
c_{<0>_{\{0\}_{1_{\{-1\}}}}}\\
&&\ot 
a_3c_{<-1>_3}c_{<0>_{\{0\}_{2_{\{-1\}}}}}\\
&\overset{(\ref{extra3})}{=}&
\varepsilon _C(c_{<0>_{\{0\}_{\{0\}_{2_{\{0\}}}}}})
\varepsilon _C(c_{<0>_{\{0\}_{\{0\}_1}}})
a_1c_{<-1>_1}
c_{<0>_{\{-1\}}}\ot 
a_2c_{<-1>_2}
c_{<0>_{\{0\}_{\{-1\}}}}\\
&&\ot 
a_3c_{<-1>_3}c_{<0>_{\{0\}_{\{0\}_{2_{\{-1\}}}}}}\\
&=&\varepsilon _C(c_{<0>_{\{0\}_{\{0\}_{\{0\}}}}})
a_1c_{<-1>_1}
c_{<0>_{\{-1\}}}\ot 
a_2c_{<-1>_2}
c_{<0>_{\{0\}_{\{-1\}}}}\ot 
a_3c_{<-1>_3}c_{<0>_{\{0\}_{\{0\}_{\{-1\}}}}}, 
\end{eqnarray*}
${\;\;\;}$$(id_A\ot \rho ')\circ (\rho '\ot id_C)\circ 
(id_A\ot \Delta _C)(a\ot c)$
\begin{eqnarray*}
&=&(id_A\ot \rho ')\circ (\rho '\ot id_C)(a\ot c_1\ot c_2)\\
&=&(id_A\ot \rho ')(\varepsilon _C(c_{1_{<0>_{\{0\}_{\{0\}}}}})
a_1c_{1_{<-1>_1}}c_{1_{<0>_{\{-1\}}}}\ot a_2c_{1_{<-1>_2}}
c_{1_{<0>_{\{0\}_{\{-1\}}}}}\ot c_2)\\
&=&\varepsilon _C(c_{1_{<0>_{\{0\}_{\{0\}}}}})
\varepsilon _C(c_{2_{<0>_{\{0\}_{\{0\}}}}})
a_1c_{1_{<-1>_1}}c_{1_{<0>_{\{-1\}}}}\\
&&\ot a_2
c_{1_{<-1>_2}}
c_{1_{<0>_{\{0\}_{\{-1\}_1}}}}
c_{2_{<-1>_1}}c_{2_{<0>_{\{-1\}}}}\\
&&\ot 
a_3
c_{1_{<-1>_3}}
c_{1_{<0>_{\{0\}_{\{-1\}_2}}}}
c_{2_{<-1>_2}}c_{2_{<0>_{\{0\}_{\{-1\}}}}}\\
&\overset{(\ref{extra2})}{=}&
\varepsilon _C(c_{<0>_{1_{\{0\}_{\{0\}}}}})
\varepsilon _C(c_{<0>_{2_{<0>_{\{0\}_{\{0\}}}}}})
a_1c_{<-1>_1}c_{<0>_{1_{\{-1\}}}}\\
&&\ot a_2
c_{<-1>_2}
(c_{<0>_{1_{\{0\}_{\{-1\}}}}}
c_{<0>_{2_{<-1>}}})_1c_{<0>_{2_{<0>_{\{-1\}}}}}\\
&&\ot 
a_3c_{<-1>_3}
(c_{<0>_{1_{\{0\}_{\{-1\}}}}}
c_{<0>_{2_{<-1>}}})_2c_{<0>_{2_{<0>_{\{0\}_{\{-1\}}}}}}\\
&\overset{(\ref{extra3})}{=}&
\varepsilon _C(c_{<0>_{\{0\}_{1_{\{0\}}}}})
\varepsilon _C(c_{<0>_{\{0\}_{2_{<0>_{\{0\}_{\{0\}}}}}}})
a_1c_{<-1>_1}c_{<0>_{\{-1\}}}\\
&&\ot a_2
c_{<-1>_2}
(c_{<0>_{\{0\}_{1_{\{-1\}}}}}
c_{<0>_{\{0\}_{2_{<-1>}}}})_1c_{<0>_{\{0\}_{2_{<0>_{\{-1\}}}}}}\\
&&\ot 
a_3c_{<-1>_3}
(c_{<0>_{\{0\}_{1_{\{-1\}}}}}
c_{<0>_{\{0\}_{2_{<-1>}}}})_2
c_{<0>_{\{0\}_{2_{<0>_{\{0\}_{\{-1\}}}}}}}\\
&\overset{(\ref{cucu2})}{=}&
\varepsilon _C(c_{<0>_{\{0\}_{\{0\}_{\{0\}}}}})
a_1c_{<-1>_1}
c_{<0>_{\{-1\}}}\ot 
a_2c_{<-1>_2}
c_{<0>_{\{0\}_{\{-1\}}}}\\
&&\ot 
a_3c_{<-1>_3}c_{<0>_{\{0\}_{\{0\}_{\{-1\}}}}}, \;\;\;q.e.d.
\end{eqnarray*}
\underline{Proof of (\ref{cobrz5})}: Let $a\in A$ and 
$c\in C$; by using a part of the computation performed 
for proving (\ref{cobrz4}), we compute:\\[2mm]
${\;\;\;}$$(W'\ot id_A)\circ (id_A\ot W')\circ (\rho '\ot id_C)\circ 
(id_A\ot \Delta _C)(a\ot c)$
\begin{eqnarray*}
&=&(W'\ot id_A)(\varepsilon _C(c_{2_{\{0\}}})
a_1c_{1_{<-1>_1}}c_{1_{<0>_{\{-1\}}}}\ot c_{1_{<0>_{\{0\}}}}\ot 
a_2c_{1_{<-1>_2}}c_{2_{\{-1\}}})\\
&=&\varepsilon _C(c_{2_{\{0\}}})
\varepsilon _C(c_{1_{<0>_{\{0\}_{<0>_{2_{\{0\}}}}}}})
c_{1_{<0>_{\{0\}_{<0>_1}}}}\ot a_1c_{1_{<-1>_1}}
c_{1_{<0>_{\{-1\}}}}c_{1_{<0>_{\{0\}_{<-1>}}}}\\
&&c_{1_{<0>_{\{0\}_{<0>_{2_{\{-1\}}}}}}}\ot 
a_2c_{1_{<-1>_2}}c_{2_{\{-1\}}}\\
&\overset{(\ref{extra2})}{=}&
\varepsilon _C(c_{<0>_{2_{\{0\}}}})
\varepsilon _C(c_{<0>_{1_{\{0\}_{<0>_{2_{\{0\}}}}}}})
c_{<0>_{1_{\{0\}_{<0>_1}}}}\ot a_1c_{<-1>_1}
c_{<0>_{1_{\{-1\}}}}c_{<0>_{1_{\{0\}_{<-1>}}}}\\
&&c_{<0>_{1_{\{0\}_{<0>_{2_{\{-1\}}}}}}}\ot 
a_2c_{<-1>_2}c_{<0>_{2_{\{-1\}}}}\\
&\overset{(\ref{extra3})}{=}&
\varepsilon _C(c_{<0>_{\{0\}_{2_{\{0\}}}}})
\varepsilon _C(c_{<0>_{\{0\}_{1_{<0>_{2_{\{0\}}}}}}})
c_{<0>_{\{0\}_{1_{<0>_1}}}}\ot a_1c_{<-1>_1}
c_{<0>_{\{-1\}}}c_{<0>_{\{0\}_{1_{<-1>}}}}\\
&&c_{<0>_{\{0\}_{1_{<0>_{2_{\{-1\}}}}}}}\ot 
a_2c_{<-1>_2}c_{<0>_{\{0\}_{2_{\{-1\}}}}}\\
&\overset{(\ref{extra2})}{=}&
\varepsilon _C(c_{<0>_{\{0\}_{2_{\{0\}}}}})
\varepsilon _C(c_{<0>_{\{0\}_{1_{2_{\{0\}}}}}})
c_{<0>_{\{0\}_{1_{1_{<0>}}}}}\ot a_1c_{<-1>_1}
c_{<0>_{\{-1\}}}c_{<0>_{\{0\}_{1_{1_{<-1>}}}}}\\
&&c_{<0>_{\{0\}_{1_{2_{\{-1\}}}}}}\ot 
a_2c_{<-1>_2}c_{<0>_{\{0\}_{2_{\{-1\}}}}}\\
&=&\varepsilon _C(c_{<0>_{\{0\}_{3_{\{0\}}}}})
\varepsilon _C(c_{<0>_{\{0\}_{2_{\{0\}}}}})
c_{<0>_{\{0\}_{1_{<0>}}}}\ot a_1c_{<-1>_1}
c_{<0>_{\{-1\}}}c_{<0>_{\{0\}_{1_{<-1>}}}}\\
&&c_{<0>_{\{0\}_{2_{\{-1\}}}}}\ot 
a_2c_{<-1>_2}c_{<0>_{\{0\}_{3_{\{-1\}}}}}\\
&\overset{(\ref{extra2})}{=}&
\varepsilon _C(c_{<0>_{\{0\}_{<0>_{3_{\{0\}}}}}})
\varepsilon _C(c_{<0>_{\{0\}_{<0>_{2_{\{0\}}}}}})
c_{<0>_{\{0\}_{<0>_1}}}\ot a_1c_{<-1>_1}
c_{<0>_{\{-1\}}}c_{<0>_{\{0\}_{<-1>}}}\\
&&c_{<0>_{\{0\}_{<0>_{2_{\{-1\}}}}}}\ot 
a_2c_{<-1>_2}c_{<0>_{\{0\}_{<0>_{3_{\{-1\}}}}}}\\
&\overset{(\ref{cros3})}{=}&
\varepsilon _C(c_{<0>_{3_{\{0\}}}})
\varepsilon _C(c_{<0>_{2_{\{0\}}}})
c_{<0>_1}\ot a_1c_{<-1>_1}
c_{<0>_{2_{\{-1\}}}}\ot 
a_2c_{<-1>_2}c_{<0>_{3_{\{-1\}}}}\\
&\overset{(\ref{extra2})}{=}&
\varepsilon _C(c_{3_{\{0\}}})
\varepsilon _C(c_{2_{\{0\}}})
c_{1_{<0>}}\ot a_1c_{1_{<-1>_1}}
c_{2_{\{-1\}}}\ot 
a_2c_{1_{<-1>_2}}c_{3_{\{-1\}}}\\
&\overset{(\ref{cucu5})}{=}&
\varepsilon _C(c_{2_{\{0\}_{\{0\}}}})
c_{1_{<0>}}\ot a_1c_{1_{<-1>_1}}
c_{2_{\{-1\}}}\ot 
a_2c_{1_{<-1>_2}}c_{2_{\{0\}_{\{-1\}}}},
\end{eqnarray*}
${\;\;\;}$$(id_C\ot \rho ')\circ (W'\ot id_C)
\circ (id_A\ot \Delta _C)(a\ot c)$
\begin{eqnarray*}
&=&(id_C\ot \rho ')\circ (W'\ot id_C)(a\ot c_1\ot c_2)\\
&=&(id_C\ot \rho ')(\varepsilon _C(c_{1_{<0>_{2_{\{0\}}}}})
c_{1_{<0>_1}}\ot ac_{1_{<-1>}}c_{1_{<0>_{2_{\{-1\}}}}}
\ot c_2)\\
&=&\varepsilon _C(c_{1_{<0>_{2_{\{0\}}}}})
\varepsilon _C(c_{2_{<0>_{\{0\}_{\{0\}}}}})
c_{1_{<0>_1}}\ot 
a_1c_{1_{<-1>_1}}c_{1_{<0>_{2_{\{-1\}_1}}}}
c_{2_{<-1>_1}}c_{2_{<0>_{\{-1\}}}}\\
&&\ot a_2c_{1_{<-1>_2}}c_{1_{<0>_{2_{\{-1\}_2}}}}
c_{2_{<-1>_2}}c_{2_{<0>_{\{0\}_{\{-1\}}}}}\\
&\overset{(\ref{extra2})}{=}&
\varepsilon _C(c_{<0>_{1_{2_{\{0\}}}}})
\varepsilon _C(c_{<0>_{2_{<0>_{\{0\}_{\{0\}}}}}})
c_{<0>_{1_1}}\\
&&\ot 
a_1c_{<-1>_1}(c_{<0>_{1_{2_{\{-1\}}}}}
c_{<0>_{2_{<-1>}}})_1c_{<0>_{2_{<0>_{\{-1\}}}}}\\
&&\ot a_2c_{<-1>_2}(c_{<0>_{1_{2_{\{-1\}}}}}
c_{<0>_{2_{<-1>}}})_2c_{<0>_{2_{<0>_{\{0\}_{\{-1\}}}}}}\\
&=&\varepsilon _C(c_{<0>_{2_{\{0\}}}})
\varepsilon _C(c_{<0>_{3_{<0>_{\{0\}_{\{0\}}}}}})
c_{<0>_1}\ot 
a_1c_{<-1>_1}(c_{<0>_{2_{\{-1\}}}}
c_{<0>_{3_{<-1>}}})_1c_{<0>_{3_{<0>_{\{-1\}}}}}\\
&&\ot a_2c_{<-1>_2}(c_{<0>_{2_{\{-1\}}}}
c_{<0>_{3_{<-1>}}})_2c_{<0>_{3_{<0>_{\{0\}_{\{-1\}}}}}}\\
&\overset{(\ref{cucu2})}{=}&
\varepsilon _C(c_{<0>_{2_{\{0\}_{\{0\}}}}})
c_{<0>_1}\ot a_1c_{<-1>_1}c_{<0>_{2_{\{-1\}}}}\ot 
a_2c_{<-1>_2}c_{<0>_{2_{\{0\}_{\{-1\}}}}}\\
&\overset{(\ref{extra2})}{=}&
\varepsilon _C(c_{2_{\{0\}_{\{0\}}}})
c_{1_{<0>}}\ot a_1c_{1_{<-1>_1}}
c_{2_{\{-1\}}}\ot 
a_2c_{1_{<-1>_2}}c_{2_{\{0\}_{\{-1\}}}}, \;\;\;q.e.d.
\end{eqnarray*}
So, $A_{\;W', \rho '}\ot C$ is indeed a crossed coproduct coalgebra. 
If we denote by $\Delta '$ its comultiplication, then by 
(\ref{formco}) we know that $\Delta '$ is given by the 
formula $\Delta '=(id_A\ot W'\ot id_C)\circ 
(\rho '\ot \Delta _C)\circ (id_A\ot \Delta _C)$. 
We claim that this formula reduces to 
\begin{eqnarray}
&&\Delta '(a\ot c)=a_1c_{1_{<-1>_1}}c_{1_{<0>_{\{-1\}}}}
\ot c_{1_{<0>_{\{0\}}}}\ot a_2c_{1_{<-1>_2}}c_{2_{\{-1\}}} 
\ot c_{2_{\{0\}}}, \label{deltaprim}
\end{eqnarray}
for all $a\in A$, $c\in C$. Indeed, we compute:
\begin{eqnarray*}
\Delta '(a\ot c)&=&(id_A\ot W'\ot id_C)\circ 
(\rho '\ot \Delta _C)(a\ot c_1\ot c_2)\\
&=&(id_A\ot W'\ot id_C)(\varepsilon _C(c_{1_{<0>_{\{0\}_{\{0\}}}}})
a_1c_{1_{<-1>_1}}c_{1_{<0>_{\{-1\}}}}\ot 
a_2c_{1_{<-1>_2}}c_{1_{<0>_{\{0\}_{\{-1\}}}}}\\
&&\ot c_2\ot c_3)\\
&=&\varepsilon _C(c_{1_{<0>_{\{0\}_{\{0\}}}}})
a_1c_{1_{<-1>_1}}c_{1_{<0>_{\{-1\}}}}\ot 
\varepsilon _C(c_{2_{<0>_{2_{\{0\}}}}})
c_{2_{<0>_1}}\\
&&\ot a_2c_{1_{<-1>_2}}c_{1_{<0>_{\{0\}_{\{-1\}}}}}
c_{2_{<-1>}}c_{2_{<0>_{2_{\{-1\}}}}}\ot c_3\\
&\overset{(\ref{extra2})}{=}&
\varepsilon _C(c_{1_{<0>_{\{0\}_{\{0\}}}}})
a_1c_{1_{<-1>_1}}c_{1_{<0>_{\{-1\}}}}\ot 
\varepsilon _C(c_{2_{2_{\{0\}}}})
c_{2_{1_{<0>}}}\\
&&\ot a_2c_{1_{<-1>_2}}c_{1_{<0>_{\{0\}_{\{-1\}}}}}
c_{2_{1_{<-1>}}}c_{2_{2_{\{-1\}}}}\ot c_3\\
&=&\varepsilon _C(c_{1_{<0>_{\{0\}_{\{0\}}}}})
a_1c_{1_{<-1>_1}}c_{1_{<0>_{\{-1\}}}}\ot 
\varepsilon _C(c_{3_{\{0\}}})
c_{2_{<0>}}\\
&&\ot a_2c_{1_{<-1>_2}}c_{1_{<0>_{\{0\}_{\{-1\}}}}}
c_{2_{<-1>}}c_{3_{\{-1\}}}\ot c_4\\
&\overset{(\ref{extra3})}{=}&
\varepsilon _C(c_{1_{<0>_{\{0\}_{\{0\}}}}})
a_1c_{1_{<-1>_1}}c_{1_{<0>_{\{-1\}}}}\ot 
\varepsilon _C(c_{3_{\{0\}_1}})
c_{2_{<0>}}\\
&&\ot a_2c_{1_{<-1>_2}}c_{1_{<0>_{\{0\}_{\{-1\}}}}}
c_{2_{<-1>}}c_{3_{\{-1\}}}\ot c_{3_{\{0\}_2}}\\
&=&\varepsilon _C(c_{1_{<0>_{\{0\}_{\{0\}}}}})
a_1c_{1_{<-1>_1}}c_{1_{<0>_{\{-1\}}}}\ot 
c_{2_{<0>}}\\
&&\ot a_2c_{1_{<-1>_2}}c_{1_{<0>_{\{0\}_{\{-1\}}}}}
c_{2_{<-1>}}c_{3_{\{-1\}}}\ot c_{3_{\{0\}}}\\
&\overset{(\ref{cucu5})}{=}&
\varepsilon _C(c_{1_{<0>_{1_{\{0\}}}}})
\varepsilon _C(c_{1_{<0>_{2_{\{0\}}}}})
a_1c_{1_{<-1>_1}}c_{1_{<0>_{1_{\{-1\}}}}}\ot 
c_{2_{<0>}}\\
&&\ot a_2c_{1_{<-1>_2}}c_{1_{<0>_{2_{\{-1\}}}}}
c_{2_{<-1>}}c_{3_{\{-1\}}}\ot c_{3_{\{0\}}}\\
&\overset{(\ref{extra2})}{=}&
\varepsilon _C(c_{1_{1_{<0>_{\{0\}}}}})
\varepsilon _C(c_{1_{2_{\{0\}}}})
a_1c_{1_{1_{<-1>_1}}}c_{1_{1_{<0>_{\{-1\}}}}}\ot 
c_{2_{<0>}}\\
&&\ot a_2c_{1_{1_{<-1>_2}}}c_{1_{2_{\{-1\}}}}
c_{2_{<-1>}}c_{3_{\{-1\}}}\ot c_{3_{\{0\}}}\\
&=&\varepsilon _C(c_{1_{<0>_{\{0\}}}})
\varepsilon _C(c_{2_{\{0\}}})
a_1c_{1_{<-1>_1}}c_{1_{<0>_{\{-1\}}}}\ot 
c_{3_{<0>}}\\
&&\ot a_2c_{1_{<-1>_2}}c_{2_{\{-1\}}}
c_{3_{<-1>}}c_{4_{\{-1\}}}\ot c_{4_{\{0\}}}\\
&\overset{(\ref{cucu2})}{=}&
\varepsilon _C(c_{1_{<0>_{\{0\}}}})
a_1c_{1_{<-1>_1}}c_{1_{<0>_{\{-1\}}}}\ot 
c_2\ot a_2c_{1_{<-1>_2}}c_{3_{\{-1\}}}\ot c_{3_{\{0\}}}\\
&\overset{(\ref{extra2})}{=}&
\varepsilon _C(c_{<0>_{1_{\{0\}}}})
a_1c_{<-1>_1}c_{<0>_{1_{\{-1\}}}}\ot 
c_{<0>_2}\ot a_2c_{<-1>_2}c_{<0>_{3_{\{-1\}}}}
\ot c_{<0>_{3_{\{0\}}}}\\
&\overset{(\ref{extra3})}{=}&
\varepsilon _C(c_{<0>_{\{0\}_1}})
a_1c_{<-1>_1}c_{<0>_{\{-1\}}}\ot 
c_{<0>_{\{0\}_2}}\ot a_2c_{<-1>_2}
c_{<0>_{\{0\}_{3_{\{-1\}}}}}\\
&&\ot c_{<0>_{\{0\}_{3_{\{0\}}}}}\\
&=&a_1c_{<-1>_1}c_{<0>_{\{-1\}}}\ot 
c_{<0>_{\{0\}_1}}\ot a_2c_{<-1>_2}
c_{<0>_{\{0\}_{2_{\{-1\}}}}}
\ot c_{<0>_{\{0\}_{2_{\{0\}}}}}\\
&\overset{(\ref{extra3})}{=}&
a_1c_{<-1>_1}c_{<0>_{1_{\{-1\}}}}\ot 
c_{<0>_{1_{\{0\}}}}\ot a_2c_{<-1>_2}
c_{<0>_{2_{\{-1\}}}}
\ot c_{<0>_{2_{\{0\}}}}\\
&\overset{(\ref{extra2})}{=}&
a_1c_{1_{<-1>_1}}c_{1_{<0>_{\{-1\}}}}\ot 
c_{1_{<0>_{\{0\}}}}\ot a_2c_{1_{<-1>_2}}
c_{2_{\{-1\}}}
\ot c_{2_{\{0\}}}, \;\;\;q.e.d.
\end{eqnarray*}
We have proved that $A_{\;W', \rho '}\ot C$ is a crossed 
coproduct coalgebra and we know from Theorem \ref{main} 
that $A\ot _{R', \sigma '}C$ is a crossed product algebra. 
We have to prove now that  
$A _{\;W'}^{\;\rho '}\bowtie _{\;R'}^{\;\sigma '}C$ is a 
cross product bialgebra, that is the maps 
$\varepsilon _A\ot \varepsilon _C:A\ot _{R', \sigma '}C
\rightarrow k$ and 
$\Delta ':A\ot _{R', \sigma '}C\rightarrow 
(A\ot _{R', \sigma '}C)\ot (A\ot _{R', \sigma '}C)$
are algebra maps. We will give an indirect proof. Define 
the map 
$\varphi :A _{\;W'}^{\;\rho '}\bowtie _{\;R'}^{\;\sigma '}C
\rightarrow 
A _{\;W_0}^{\;\rho _0}\bowtie _{\;R}^{\;\sigma }C$, 
$\varphi (a\ot c)=ac_{<-1>}\ot c_{<0>}$. We know from 
Theorem \ref{main} that $\varphi $ is an algebra isomorphism. 
We compute: 
\begin{eqnarray*}
(\varepsilon _A\ot \varepsilon _C)\circ \varphi (a\ot c)&=&
\varepsilon _A(ac_{<-1>})\varepsilon _C(c_{<0>})\\
&=&\varepsilon _A(a)\varepsilon _A(c_{<-1>})
\varepsilon _C(c_{<0>})\\
&\overset{(\ref{extra1})}{=}&
\varepsilon _A(a)\varepsilon _C(c)\\
&=&(\varepsilon _A\ot \varepsilon _C)(a\ot c).
\end{eqnarray*}
So $(\varepsilon _A\ot \varepsilon _C)\circ \varphi =
\varepsilon _A\ot \varepsilon _C$, and since 
$\varepsilon _A\ot \varepsilon _C:A\ot _{R, \sigma }C
\rightarrow k$ is an algebra map (because 
$A _{\;W_0}^{\;\rho _0}\bowtie _{\;R}^{\;\sigma }C$ is a 
cross product bialgebra) and $\varphi $ is an algebra 
isomorphism, it follows that 
$\varepsilon _A\ot \varepsilon _C:A\ot _{R', \sigma '}C
\rightarrow k$ is an algebra map. If we consider 
the comultiplication $\Delta $ of 
$A _{\;W_0}^{\;\rho _0}\bowtie _{\;R}^{\;\sigma }C$, 
which is defined by the formula $\Delta (a\ot c)=
a_1\ot c_1\ot a_2\ot c_2$ and which is an algebra map 
from $A\ot _{R, \sigma }C$ to $(A\ot _{R, \sigma }C)\ot 
(A\ot _{R, \sigma }C)$ because 
$A _{\;W_0}^{\;\rho _0}\bowtie _{\;R}^{\;\sigma }C$ is a 
cross product bialgebra, it is very easy to see, by using the 
formula (\ref{deltaprim}), that we have 
$\Delta '=(\varphi ^{-1}\ot \varphi ^{-1})\circ \Delta \circ \varphi $, 
and since $\varphi $ is an algebra isomorphism it follows that 
$\Delta ':A\ot _{R', \sigma '}C\rightarrow 
(A\ot _{R', \sigma '}C)\ot (A\ot _{R', \sigma '}C)$ is an 
algebra map. 

So, $A _{\;W'}^{\;\rho '}\bowtie _{\;R'}^{\;\sigma '}C$
is indeed a cross product bialgebra, and the relations 
$(\varepsilon _A\ot \varepsilon _C)\circ \varphi =
\varepsilon _A\ot \varepsilon _C$ and 
$\Delta '=(\varphi ^{-1}\ot \varphi ^{-1})\circ 
\Delta \circ \varphi $ show also that the map $\varphi $ 
is a coalgebra isomorphism, hence it is a bialgebra 
isomorphism. We know that $\varphi $ satisfies 
$\varphi (a\ot 1_C)=a\ot 1_C$, for all $a\in A$, hence 
$\varphi $ is a morphism of left $A$-modules, and the 
relation (\ref{extra2}) implies that $\varphi $ is also a 
morphism of right $C$-comodules. Hence, the cross 
product bialgebras 
$A _{\;W'}^{\;\rho '}\bowtie _{\;R'}^{\;\sigma '}C$ and
$A _{\;W_0}^{\;\rho _0}\bowtie _{\;R}^{\;\sigma }C$ 
are indeed equivalent.

Conversely, assume that 
$A _{\;W'}^{\;\rho '}\bowtie _{\;R'}^{\;\sigma '}C$ 
is a cross product bialgebra equivalent to 
$A _{\;W_0}^{\;\rho _0}\bowtie _{\;R}^{\;\sigma }C$, 
say via a map  
$\varphi :A _{\;W'}^{\;\rho '}\bowtie _{\;R'}^{\;\sigma '}C
\simeq
A _{\;W_0}^{\;\rho _0}\bowtie _{\;R}^{\;\sigma }C$. 
By Theorem \ref{mainconverse}, there exist linear maps 
$\theta , \gamma :C\rightarrow A\ot C$, 
with notation $\theta (c)=c_{<-1>}\ot c_{<0>}$ and 
$\gamma (c)=c_{\{-1\}}\ot c_{\{0\}}$, for all $c\in C$, such that 
the conditions (\ref{Rprim})--(\ref{cros4}) are satisfied and 
moreover we have 
$\varphi (a\ot c)=ac_{<-1>}\ot c_{<0>}$ and 
$\varphi ^{-1}(a\ot c)=ac_{\{-1\}}\ot c_{\{0\}}$ for all 
$a\in A$ and $c\in C$. By writing down the conditions that 
$\varphi $ and $\varphi ^{-1}$ are morphisms of comodules, we obtain the 
conditions (\ref{extra2}) and (\ref{extra3}). By writing down the conditions that 
$\varphi $ and $\varphi ^{-1}$ are counital, we obtain the condition (\ref{extra1}). 
By writing down 
the condition that $\varphi $ is a coalgebra isomorphism, 
we obtain $\Delta '=(\varphi ^{-1}\ot \varphi ^{-1})\circ 
\Delta \circ \varphi $, where $\Delta $ is the comultiplication of 
$A _{\;W_0}^{\;\rho _0}\bowtie _{\;R}^{\;\sigma }C$ and 
$\Delta '$ is the comultiplication of 
$A _{\;W'}^{\;\rho '}\bowtie _{\;R'}^{\;\sigma '}C$, hence 
$\Delta '$ is defined by the formula (\ref{deltaprim}). Finally, 
from Remark \ref{tare}, we can obtain the formulae for 
$W'$ and $\rho '$ by using the formula for $\Delta '$, and an easy 
computation using the formulae (\ref{extra2}) and (\ref{extra3}) 
shows that  $W'$ and $\rho '$ are given by the formulae 
(\ref{Wprim}) and respectively (\ref{rhoprim}).
\end{proof}

We present an example of equivalent cross product bialgebras, 
both of the type $A _{\;W}^{\;\rho}\bowtie _{\;R}^{\;\sigma }C$ 
where $A$ is a bialgebra and having the tensor product coalgebra structure. 
In \cite{am}, the authors introduced a construction, called {\em unified product}, 
which is a bialgebra of the type $A\ltimes H$, where $A$ is a bialgebra and $H$ is a 
coalgebra (with extra structures) and having the tensor product coalgebra structure. 
One can easily see that a unified product $A\ltimes H$ is a cross product bialgebra 
$A _{\;W_0}^{\;\rho _0}\bowtie _{\;R}^{\;\sigma }H$ 
(for certain $R$ and $\sigma $). If $A$ is a Hopf algebra and two unified 
products $A\ltimes H$ and $A\ltimes 'H$ are given such that they are both 
Hopf algebras, then \cite{am}, Theorem 3.4 provides a characterization of when 
$A\ltimes H$ and $A\ltimes ' H$ are (in our terminology) equivalent. 
If $A$, $A\ltimes H$ and $A\ltimes 'H$ are only bialgebras, Theorem 3.4 in 
\cite{am} does not work anymore; but one can characterize when $A\ltimes H$ and 
$A\ltimes 'H$ are equivalent in this situation by using the above Theorem  
\ref{maincros}.  \\[1mm]

There exists a ''mirror version'' of Brzezi\'{n}ski's crossed product (to be studied in 
detail elsewhere), as follows: if $(B, \mu , 1_B)$ is an (associative unital) algebra and 
$W$ is a vector space equipped with a distinguished element $1_W\in W$, then 
the vector space $W\otimes B$ is an associative algebra with unit $1_W\otimes 1_B$ 
and whose multiplication   has the property that 
$(w\ot b)(1_W\ot b')=w\ot bb'$, for all $b, b'\in B$ and 
$w\in W$, if and only if there exist linear maps 
$\nu :W\ot W\rightarrow W\ot B$ and 
$P:B\ot W\rightarrow W\ot B$ satisfying a certain list of conditions. 
The corresponding algebra structure is denoted by 
$W\overline{\ot}_{P, \nu }B$.  Similarly, there exist ''mirror versions'' 
for crossed coproduct coalgebras (of the type $D_{U, \eta }\overline{\otimes}Y$, 
where  $D$ is a coalgebra and $Y$ is a vector space) and for cross 
product bialgebras, of the type 
$D _{\;U}^{\;\eta}\overline{\bowtie} _{\;P}^{\;\nu }B$. Such a cross 
product bialgebra becomes canonically a right $B$-module and a left 
$D$-comodule, we have an analogous concept of equivalence between two such 
cross product bialgebras and an analogue of Theorem \ref{maincros} for them. 

We present now an example of two equivalent cross product bialgebras 
in the above sense, of the type appearing in the ''mirror version'' of Theorem 
\ref{maincros} (that is, one of them is of the form 
$D _{\;U}^{\;\eta}\overline{\bowtie} _{\;P}^{\;\nu }B$, where $B$ is a bialgebra 
and the coalgebra structure is the tensor product coalgebra $D\otimes B$). 
Namely, a well known result of Majid (\cite{majid}) asserts that the Drinfeld double of 
a finite dimensional quasitriangular Hopf algebra $H$ is isomorphic to a certain 
Radford biproduct. The explicit Radford biproduct that appears and the 
explicit isomorphism depend on the chosen explicit realization of the Drinfeld double 
of $H$. Unlike \cite{majid}, we choose a realization of $D(H)$ on 
$H^{* cop}\otimes H$, so the concrete formulae below are different from the 
ones in \cite{majid}; the same realization of the Drinfeld double 
was chosen in \cite{jlpvo}, where it was proved that, at the algebra level, 
the isomorphism in Majid's theorem is an example of invariance under twisting 
for twisted tensor products of algebras.   

So, let $H$ be a finite dimensional Hopf algebra with antipode $S$. 
Recall that the Drinfeld double $D(H)$ is a Hopf algebra having
$H^{* cop}\ot H$ as coalgebra structure and multiplication
$(p\ot h)(p'\ot h')=p(h_1\rightharpoonup p'\leftharpoonup
S^{-1}(h_3))\ot h_2h'$,
for all $p, p'\in H^*$ and $h, h'\in H$, where $\rightharpoonup $
and $\leftharpoonup $ are the left and right regular actions of $H$
on $H^*$ given by $(h\rightharpoonup p)(h')=p(h'h)$ and
$(p\leftharpoonup h)(h')= p(hh')$. 

Assume that $r=r^1\otimes r^2\in H\otimes H$ is a quasitriangular structure 
on $H$. Then on the vector space $H^*$ one can define certain structures that 
turn $H^*$ into a bialgebra in the Yetter-Drinfeld category $_H^H{\mathcal YD}$ 
(we denote it by $\underline{H}^*$) and we have a bialgebra isomorphism 
$\varphi :\underline{H}^*\times H\simeq D(H)$, $\varphi (p\otimes h)=
p\leftharpoonup S^{-1}(r^1)\otimes r^2h$, where 
$\underline{H}^*\times H$ is the Radford biproduct (cf. \cite{radford}). 

One can  see that both $\underline{H}^*\times H$ and $D(H)$ are ''mirror versions'' 
of cross product bialgebras between $H^{* cop}$ and $H$, and for 
$D(H)=H^{* cop}\otimes H$ we have that $H$ is a bialgebra and the 
coalgebra structure is the tensor product coalgebra. Obviously, we have 
$\varphi (\varepsilon \otimes h)=\varepsilon \otimes h$, for all $h\in H$, 
so $\varphi $ is a morphism of right $H$-modules, and one can easily check 
that $\varphi $ is also  a morphism of left $H^{* cop}$-comodules. Hence, 
$\underline{H}^*\times H$ and $D(H)$ are equivalent.   

\end{document}